\documentclass[12pt, reqno]{amsart}
\usepackage{amsmath, amsthm, amscd, amsfonts, amssymb, graphicx, color}
\usepackage{enumerate}
\usepackage[bookmarksnumbered, colorlinks, plainpages]{hyperref}
\hypersetup{colorlinks=true,linkcolor=red, anchorcolor=green, citecolor=cyan, urlcolor=red, filecolor=magenta, pdftoolbar=true}

\newtheorem{theorem}{Theorem}[section]
\newtheorem{lemma}[theorem]{Lemma}

\newtheorem{corollary}[theorem]{Corollary}
\theoremstyle{definition}

\newtheorem{example}[theorem]{Example}

\newtheorem{remark}[theorem]{Remark}
\numberwithin{equation}{section}

\begin{document}

\setcounter{page}{1}

\title[Short Title]{projections  in  Toeplitz algebra}

\author[Hui Dan, Xuanhao Ding, Kunyu Guo, \MakeLowercase{and} Yuanqi Sang]{Hui Dan$^{1}$, Xuanhao Ding$^2$, Kunyu Guo$^3$,\MakeLowercase{and} Yuanqi Sang$^{4*}$ }

\address{$^{1}$ College of Mathematics, Sichuan University, Chengdu, Sichuan
610065, P.R. China.}
\email{\textcolor[rgb]{0.00,0.00,0.84}{danhuimath@gmail.com}}

\address{$^{2}$School of Mathematics and Statistics, Chongqing Technology and Business University,
 Chongqing 400067, P.R. China.}
\email{\textcolor[rgb]{0.00,0.00,0.84}{xuanhaod@qq.com}}

\address{$^{3}$School of Mathematical Sciences, Fudan University, Shanghai 200433, P.R. China.}
\email{\textcolor[rgb]{0.00,0.00,0.84}{kyguo@fudan.edu.cn}}

\address{$^{3}$School of Economic Mathematics, Southwestern University of Finance and Economics, Chengdu 611130,  P.R. China.}
\email{\textcolor[rgb]{0.00,0.00,0.84}{sangyq@swufe.edu.cn}}

\let\thefootnote\relax\footnote{}



\date{.
\newline \indent $^{*}$Corresponding author}

\begin{abstract}
Motivated by  Barr{\'\i}a-Halmos's \cite[Question 19]{barria1982asymptotic} and Halmos's \cite[Problem 237]{Halmos1978A}, we explore projections in Toeplitz algebra on the Hardy space. We show that the product of two Toeplitz (Hankel) operators is a projection if
and only if it is the projection onto one of the invariant subspaces of the shift (backward shift) operator. As a consequence one obtains new proofs  of criterion for   Toeplitz  operators and Hankel operators to be partial isometries. Furthermore, we completely characterize when the self-commutator of a Toeplitz operator is a projection. This provides a class of nontrivial projections in Toeplitz algebra.
\end{abstract} \maketitle

\section{Introduction }
Let  $\mathbb{D}$  be the open disk in the complex plane and $\mathbb{T}$  its boundary.
The Hardy space $H^{2}$ is  the subspace of $L^{2}=L^{2}(\mathbb{T})$ consisting of functions whose Fourier coefficients corresponding to  negative integers vanish.
A function $\vartheta\in H^{2}$ is
called an inner function if $|\vartheta(e^{i\theta})|=1$ a.e.

For $\varphi$ in $L^{\infty}=L^{\infty}(\mathbb{T}),$ the Toeplitz operator $T_{\varphi}$ with symbol $\varphi$ and the Hankel operator $H_{\varphi}$ with symbol $\varphi$ are defined on $H^{2}$
as the following:
\begin{align*}
T_{\varphi}f&=P(\varphi f),\\
H_{\varphi}f&=(I-P)(\varphi f),\quad f\in H^{2},
\end{align*}
where $P$ is the orthogonal projection of $L^{2}$ onto $H^{2}.$
The Toeplitz algebra $\mathfrak{T}_{L^{\infty}}$
is the $C^{*}-$algebra  generated by  $\{T_{\phi},\phi\in L^{\infty}\}.$
We say that a bounded operator $Q$ on a Hilbert space is a projection if $Q$ satisfies \[Q=Q^{*}=Q^{2}.\]

The study of projections, and  applications of such study to illuminate structure of $C^{*}-$algebras, have been an enduring theme in operator algebra.
In particular,
progresses on projections in Toeplitz algebra will
shed new light on the structure of $\mathfrak{T}_{L^{\infty}}$, for instance, compact perturbation or essential commutant problem \cite{axler1978products,davidson1977operators,gu2004operators,chen2005compact,xia2008essential}, when a Hankel operator is in
$\mathfrak{T}_{L^{\infty}}$\cite{barria1996hankel,xiaoman1990hankel}, is Ces\`{a}ro operator in $\mathfrak{T}_{L^{\infty}}$\cite{martinez2002essentially}?.etc.

In \cite[Question 19]{barria1982asymptotic}, J. Barr{\'\i}a and P. R. Halmos  raised
a problem:
\begin{center}
``Which projections belong to $\mathfrak{T}_{L^{\infty}}$?"
\end{center}
They remarked that although the question is vague, it ``might give a hint to a suitably general context in which Toeplitz theory can be embedded", and in which problems in Toeplitz theory become ``more manageable".
To better understand this problem, we first observe that if
$T$ is a finite rank diagonal operator with diagonal entries equal to 0 or 1,
then $T$ belongs to $\mathfrak{T}_{L^{\infty}},$ by the formula $I-T_{z^{n+1}}T_{\bar{z}^{n+1}}=z^{n}\otimes z^{n}(n\geq 0).$ Are there any other projections in $\mathfrak{T}_{L^{\infty}}?$
For a unital $C^{*}$-algebra, the projections $0$ and $I$ are trivial.
The purpose of the current paper is find more nontrivial projections in $\mathfrak{T}_{L^{\infty}},$ and classify them is some sense.

It is easy to see that all finite sums of finite products of Toeplitz operators form
a dense set in $\mathfrak{T}_{L^{\infty}}.$ For J. Barr{\'\i}a and P. R. Halmos' problem, we should find a condition for the operator $\sum_{i=1}^{m}\prod_{j=1}^{n}T_{\varphi_{ij}}$ to be a projection. According to the  solving process of zero product problem of Toeplitz operators
\cite{guo1996problem,gu2000products,aleman2009zero}, we think that it maybe difficult when $n$ and $m$ are large.
S. Axler made an important observation in \cite[(14)]{barria1982asymptotic}:
the projection onto a invariant subspace of $T_{z}$ belongs to $\mathfrak{T}_{L^{\infty}}$, and by Beurling's theorem, it equals $T_{\theta}T_{\bar{\theta}}$ for some inner function $\theta.$
Inspired by this,
we will initially consider that for which functions $f$ and $g$,
$T_{f}T_{g}$ is a projection?
In section 3, we find  that if
$T_{f}T_{g}$ is a projection, it must
be the projection onto a invariant subspace of $T_{z}.$
This result covers the result of A. Brown and R. Douglas in \cite{brown1965partially}.

The central role in this work is played by the following theorem (see\cite[7.11]{douglas2012banach}or\cite[Theorem 2]{Engli1995Toeplitz}):


\noindent{\bf{Symbol mapping.}}\,\,\,
Every operator in $\mathfrak{T}_{L^{\infty}}$ is of the form
\begin{align*}
T=T_{f}+S,\quad f\in L^{\infty},  S\in \mathfrak{S}
\end{align*}
where $\mathfrak{S}$ is the semicommutator ideal generated by all semicommutators $T_{fg}-T_{f}T_{g},f,g\in L^{\infty}.$

Since a Toeplitz operator is a projection if and only if it is $0$ or $I$ \cite[Corollary 5]{brown1964algebraic}.
In the view of the symbol mapping theorem and the following important formula
\begin{align}\label{HT}
T_{fg}-T_{f}T_{g}=H^{*}_{\bar{f}}H_{g},\quad
f,g\in L^{\infty},
\end{align}
in what follows we shall
consider that for which functions $f$ and $g$,
$H^{*}_{\bar{f}}H_{g}$ is a projection?

Let $\vartheta$ be a nonconstant inner function, the corresponding model space $K_{\vartheta}^{2}$ is defined to be \[K_{\vartheta}^{2}=H^{2}\ominus \vartheta H^{2}.\]
Moreover, $K_{\vartheta}^{2}$ is a nontrivial invariant subspace of $T^{*}_{z}$.
In section 4, we show that if
$H^{*}_{\bar{f}}H_{g}$ is a projection, then it must be a projection onto a model space. This result covers
the  decription of the partially isometric Hankel operators \cite[Theorem 2.6]{Peller2003Hankel}.

For an operator $T$ on a separable Hilbert space $\mathcal{H},$
the self-commutator of $T$ is define by $T^{*}T-TT^{*}.$
The study of self-commutator has attracted much interest. For example,
every self-adjoint operator on an infinite dimensional Hilbert space is the sum of
two self-commutators \cite{halmos1952commutators} and Berger-Shaw's theorem \cite{berger1973selfcommutators}, etc.
P. R. Halmos \cite[Problem 237]{Halmos1978A} asked that can $T^{*}T-TT^{*}$ be a projection,
and, if so, how? He also proved that if $T$ is an abnormal operator (i.e., operators that have no normal direct summands) and $\|T\|=1$, such that self-commutator of $T$ is a projection, then $T$ is an isometry.
It is still an interesting  question for Toeplitz operator.
Note that the self-commutator of $T_{f}$ is in   $\mathfrak{T}_{L^{\infty}}.$

In section 5,
we give the necessary and sufficient condition for the self-commutator of $T_{f}$ to be a projection when $T_{f}$ remains unrestricted.
There are several difficulties in proving this result. One is that the symbol mapping theorem is fail to get the information of symbol $f$, since the  corresponding symbol of $T^{*}_{f}T_{f}-T_{f}T^{*}_{f}$ is zero. Another is to obtain the range of $T^{*}_{f}T_{f}-T_{f}T^{*}_{f}.$
We overcome these obstacles by linking
hyponormal Toeplitz operators and truncated Toeplitz operators.

In section 6, we describe the $C^{*}-$algebra generated by $T_{u}T_{\bar{u}}$ for all inner functions $u.$
We can now state our main results.

\noindent{\bf{Theorem \ref{main1}}}
If $f,g\in L^{\infty}(\mathbb{T}),$ then the following statements
are equivalent.
\begin{enumerate}
\item $T_{f}T_{g}$ is a
nontrivial projection;

\item $T_{f}T_{g}$ is a
projection, and its range is a nontrivial invariant subspace of the shift operator $T_{z};$

\item There exist a nonconstant inner function $\theta$ and a nonzero constant $a$ such that $f=a\theta$ and $ g=\frac{\bar{\theta}}{a}$.

\end{enumerate}

\noindent{\bf{Theorem \ref{2H}}}
If $f,g\in L^{\infty}(\mathbb{T}),$ then the following statements
are equivalent.
\begin{enumerate}

\item $H^{*}_{\bar{f}}H_{g}$ is a
nontrivial projection operator;

\item The range of $H^{*}_{\bar{f}}H_{g}$ is
a model space  $K^{2}_{\theta},$ where $\theta$ is an inner function;

\item$\bar{f}+\bar{\mu}\bar{\theta},g+\frac{\bar{\theta}}{\mu}\in H^{2},$ where $\mu \in\mathbb{C}\setminus\{0\}$.
\end{enumerate}

\noindent{\bf{Theorem \ref{main3}}}
 If $\varphi\in L^{\infty}(\mathbb{T}),$ then
$T^{*}_{\varphi}T_{\varphi}-T_{\varphi}T^{*}_{\varphi}$ is a nontrivial projection operator
if and only if  one of following conditions holds
\begin{enumerate}
\item The range of $T^{*}_{\varphi}T_{\varphi}-T_{\varphi}T^{*}_{\varphi}$
    is a model space, and
    $\varphi=a\theta+b\bar{\theta}+c,$ where $\theta$ is an inner function, $a,b$ and $c$ are constant with $|a|^{2}-|b|^{2}=1;$
\item The range of $T^{*}_{\varphi}T_{\varphi}-T_{\varphi}T^{*}_{\varphi}$
    is not a model space, and
$\varphi=uv+\bar{v}+c,$ where $u$ is inner, $c$ is constant, $v\in H^{2}$ with $|v|^{2}=Re(uh+1)(h\in H^{2})$.
\end{enumerate}

\section{Self-adjointness of $T_{f}T_{g}+T_{\phi}T_{\psi}$ }

As a preparation,
we obtain a necessary and sufficient condition for self-adjointness of $T_{f}T_{g}+T_{\phi}T_{\psi}$. The main tool is finite rank operators.

Given vectors $f$ and $g$ in a separable Hilbert space $\mathcal{H},$
we define the rank-one operator $f\otimes g$ mapping $\mathcal{H}$ into itself by
\begin{align}\label{r1f}
(f\otimes g)h=\left\langle h,g\right\rangle f.
\end{align}
The following properties of rank-one
operators are  well known.

\begin{lemma}\label{rank10}
Given vectors $f$ and $g$ in a separable Hilbert space $\mathcal{H}.$
\begin{enumerate}
\item  If $f\otimes g=0$ if and only if
either $f=0$ or $g=0;$

\item   $(f\otimes g)^{*}=g\otimes f;$

\item For bounded operators $A$ and $B,$
$A(f\otimes g)B=(Af)\otimes (B^{*}g).$
\end{enumerate}
\end{lemma}

\begin{lemma}\label{rank1ss}
Given vectors $f$ and $g$ in a separable Hilbert space.
If  nonzero operator $f\otimes g$ is self-adjoint if and only if there is a nonzero real constant $\lambda$ such that $f=\lambda g.$
\end{lemma}
\begin{proof}
Assume that $f\otimes g$ is self-adjoint, we have
$f\otimes g=g\otimes f,$ and therefore
\begin{align*}
(f\otimes g)g&=(g\otimes f)g,\\
\left\langle g,g\right\rangle f&=\left\langle g,f\right\rangle g,\\
f&=\frac{\left\langle g,f\right\rangle}{\left\langle g,g\right\rangle}g.
\end{align*}
If $\left\langle g,f\right\rangle=0,$ then $f$ is the zero vector. By Lemma \ref{rank10}(1),  this contradict that $f\otimes g$ is a nonzero operator.
Let $\lambda=\frac{\left\langle g,f\right\rangle}{\left\langle g,g\right\rangle}\neq0.$
Substituting $f=\lambda g$ into $f\otimes g=g\otimes f,$
\begin{align}\label{converse}
\lambda g\otimes g=\bar{\lambda} g\otimes g.
\end{align}
Hence, $\lambda$ is a nonzero real number.
The converse follows easily from \eqref{converse}.
\end{proof}

\begin{lemma}\label{rank20}
Given vectors $f,g,\phi$ and $\psi$ in a separable Hilbert space.
If operator $f\otimes g+\phi\otimes\psi$  is zero if and only if
one of following statement hold
\begin{enumerate}
\item either $f$ or $g$ is the zero vector and either $\phi$ or $\psi$ is the zero vector;

\item $f,g,\phi$ and $\psi$ are all nonzero vectors,  $f=\lambda \phi$ and $\psi=-\bar{\lambda} g,\lambda $ is a nonzero constant.
\end{enumerate}
\end{lemma}
\begin{proof}
If one of four vectors $f,g,\phi$ and $\psi$ is zero, it is easy
to see  condition (1) hold, by Lemma \ref{rank10}(1).

Suppose that $f,g,\phi$ and $\psi$ are all nonzero vectors and
$f\otimes g=-\phi\otimes\psi,$ we have
\begin{align*}
(f\otimes g)g&=-(\phi\otimes\psi)g\\
\left\langle g,g\right\rangle f&=-\left\langle g,\psi\right\rangle \phi\\
f&=-\frac{\left\langle g,\psi\right\rangle}{\left\langle g,g\right\rangle}\phi.
\end{align*}
Let $\lambda=-\frac{\left\langle g,\psi\right\rangle}{\left\langle g,g\right\rangle},$
since $f$ is a nonzero vector, $\lambda\neq0.$
Write $f=\lambda\phi,$
we have,
\begin{align*}
f\otimes g+\phi\otimes\psi
&=\lambda\phi\otimes g+\phi\otimes\psi\\
&=\phi\otimes (\bar{\lambda}g+\psi)=0.
\end{align*}
Since $\phi$ is a nonzero vector and Lemma \ref{rank10}, $\bar{\lambda}g+\psi=0.$
It is easy to check that the converse is true.
\end{proof}
\begin{lemma}\label{rank2linear}
Given vectors $f,g,\phi$ and $\psi$ in a separable Hilbert space $\mathcal{H}.$
If $f\otimes g+\phi\otimes\psi$ is
self-adjoint, then
$\{f,g\}$ is linearly dependent  if and only if
$\{\phi,\psi\}$ is linearly dependent.
\end{lemma}
\begin{proof}
If one of $\{f,g,\phi,\psi\}$ is a nonzero vector, by Lemma \ref{rank1ss},  $\{f,g\}$  and
$\{\phi,\psi\}$ are both linearly dependent.

Assume that $f,g,\phi$ and $\psi$  are four nonzero vectors
and  $\{f,g\}$ is linearly dependent,  then  there exist a nonzero  constant $\lambda,$ such that
\begin{align}\label{fgl}
f=\lambda g.
\end{align}
Since $f\otimes g+\phi\otimes\psi$ is
self-adjoint,
\begin{align}\label{fgs}
f\otimes g+\phi\otimes\psi=g\otimes f+\psi\otimes\phi,
\end{align}
Substituting \eqref{fgl} into  \eqref{fgs}, we have
\begin{align}\label{lambdab}
(\lambda-\bar{\lambda})g\otimes g=\psi\otimes\phi-\phi\otimes\psi.
\end{align}
If $\lambda$ is real, Lemma \ref{rank10}(2) implies $\psi\otimes\phi$ is self-adjoint,
by Lemma \ref{rank1ss}, we have $\{\phi,\psi\}$ is linearly dependent.

When $\lambda\neq\bar{\lambda},$
assume that $\{\phi,\psi\}$ is linearly independent, by Gram-Schmidt procedure, there exist  two  nonzero vectors $x$  and $y$ such that
\begin{align*}
\langle x,\phi\rangle=1,
\langle x,\psi\rangle&=0,\\
\langle y,\psi\rangle=1,
\langle y,\phi\rangle&=0.
\end{align*}
Applying operator equation \eqref{lambdab} to $x$ and $y$ give
\begin{align*}
(\lambda-\bar{\lambda})\langle x,g\rangle g=&\psi,\\
(\lambda-\bar{\lambda})\langle y,g\rangle g=&-\phi.
\end{align*}
Since  $\phi$ and $\psi$  are nonzero vectors,
\begin{align*}
(\lambda-\bar{\lambda})\langle x,g\rangle &\neq 0,\\
(\lambda-\bar{\lambda})\langle y,g\rangle &\neq 0.
\end{align*}
This contradicts our assumption ($\{\phi,\psi\}$ is linearly independent). The rest of proof is the same as the above reasoning.
\end{proof}
\begin{lemma}\label{rank2s}
Given nonzero vectors $f,g,\phi$ and $\psi$ in a separable Hilbert space.
$f\otimes g+\phi\otimes\psi$ is a nonzero
self-adjoint operator if and only if
one of following statement holds
\begin{enumerate}
\item $f=\lambda g$ and $\phi=\mu \psi,$  where $\lambda,\mu\in\mathbb{R}\setminus\{0\};$
\item $f=\lambda g,\phi=\mu \psi,$ and $\psi=-ag,$ where $\lambda,\mu,a\in\mathbb{C}\setminus\{0\},Im(\lambda)\neq 0,Im(\mu)\neq 0,$ $|a|^{2}\frac{Im(\mu)}{Im(\lambda)}=-1.$
\item Both $\{f,g\}$ and $\{\phi,\psi\}$ are linearly independent,
\begin{align*}
\phi&=a_{11} f+a_{12}g\\
\psi&=a_{21} f+a_{22}g,
\end{align*}
where $a_{11},a_{12},a_{21},a_{22}\in\mathbb{C}.$
$a_{11}\bar{a}_{21},\bar{a}_{12}a_{22}\in \mathbb{R},$
$\bar{a}_{12}a_{21}-a_{11}\bar{a}_{22}=1.$
\end{enumerate}
\end{lemma}
\begin{proof}
By Lemma \ref{rank2linear},
there are two cases to consider.

\noindent{\bf{Case I}}

 Assume that $\{f,g\}$ and $\{\phi,\psi\}$ are both linearly dependent, there are two nonzero constants
$\lambda$ and $\mu$ such that
\begin{align}\label{lm}
f=\lambda g,\quad \phi=\mu \psi.
\end{align}
Since $f\otimes g+\phi\otimes\psi$  is self-adjoint,
\begin{align}\label{ss}
f\otimes g+\phi\otimes\psi=g\otimes f+\psi\otimes\phi
\end{align}
Substituting \eqref{lm} into  \eqref{ss}, we have
\begin{align}\label{gpsi}
(\lambda-\bar{\lambda})g\otimes g=(\bar{\mu}-\mu)\psi\otimes\psi.
\end{align}
This means that
$\lambda=\bar{\lambda}$ if and only if $\bar{\mu}=\mu.$
If $Im(\lambda)=Im(\mu)=0,$  then
\begin{align*}
f\otimes g+\phi\otimes\psi=\lambda g\otimes g+\mu\phi\otimes\phi,
\end{align*}
and $f\otimes g+\phi\otimes\psi$ is a self-adjoint operator.
If  $Im(\lambda)$ and $Im(\mu)$ both are non zero,
\eqref{gpsi} becomes
\begin{align*}
g\otimes g+\frac{Im(\mu)}{Im(\lambda)}\psi\otimes\psi=0.
\end{align*}
By Lemma \ref{rank20}, we have
\begin{align*}
g=\bar{a}\frac{Im(\mu)}{Im(\lambda)}\psi,\quad \psi=-ag,\quad a\in \mathbb{C}\setminus\{0\},
\end{align*}
and $|a|^{2}\frac{Im(\mu)}{Im(\lambda)}=-1.$

\noindent{\bf{Case II}}

If  both $\{f,g\}$ and $\{\phi,\psi\}$ are linearly independent,
by Gram-Schmidt procedure, there exist  two  nonzero vectors $x$  and $y$ such that
\begin{align*}
\langle y,\psi\rangle=1,
\langle y,\phi\rangle&=0,\\
\langle x,\phi\rangle=1,
\langle x,\psi\rangle&=0.
\end{align*}
Applying operator equation \eqref{ss} to $x$ and $y$ give
\begin{align*}
\phi&=-\langle y,g\rangle f+\langle y,f\rangle g, \\
\psi&=\langle x,g\rangle f-\langle x,f\rangle g.
\end{align*}
Let $a_{11}=-\langle y,g\rangle,a_{12}=\langle y,f\rangle,$
$a_{21}=\langle x,g\rangle$ and $a_{22}=-\langle x,f\rangle.$ Write
\begin{equation}{\label{2.3}}
\begin{split}
\binom{\phi}{\psi}
&=\begin{pmatrix}a_{11}& a_{12}\\
   a_{21} &   a_{22}  \end{pmatrix}\binom{f}{g}.
\end{split}
\end{equation}
Substituting \eqref{2.3} into  \eqref{ss}, we have
\begin{align*}
&f\otimes g+(a_{11}f+a_{12}g)\otimes(a_{21} f+a_{22}g)\\
=&g \otimes f+(a_{21} f+a_{22}g)\otimes(a_{11}f+a_{12}g).
\end{align*}
After simplifying we get
\begin{align*}
&\big((a_{11}\bar{a}_{21}-\bar{a}_{11}a_{21})f
+(a_{12}\bar{a}_{21}-\bar{a}_{11}a_{22}-1)g\big)\otimes f\\
=&\big((\bar{a}_{12}a_{22}-a_{12}\bar{a}_{22})g
+(\bar{a}_{12}a_{21}-a_{11}\bar{a}_{22}-1)f\big)\otimes g.
\end{align*}
Since $\{f,g\}$ is linearly independent and Lemma \ref{rank20},
\begin{align*}
a_{11}\bar{a}_{21}-\bar{a}_{11}a_{21}=&0,\\
\bar{a}_{12}a_{22}-a_{12}\bar{a}_{22}=&0,\\
\bar{a}_{12}a_{21}-a_{11}\bar{a}_{22}=&1.
\end{align*}
The converse follows immediately from the above reasoning.
\end{proof}
Define an operator $V$ on $L^{2}$ by
\[Vf(w)=\overline{w}\overline{f(w)}\]
for $f\in L^{2}.$ It is easy to check that $V$ is anti-unitary. The operator $V$ satisfies
the following properties \cite[Lemma 2.1]{xia1997products}:
\begin{align*}
V^{2}&=I, \\
V P V&=(I-P),\\
V H_{f} V&=H_{f}^{*}.
\end{align*}

\begin{lemma}\label{rank1s}
If $f$ and $g$ are in $L^{\infty},$ then
\begin{align*}
T_{\bar{z}}T_{f}T_{g}T_{z}-T_{f}T_{g}
=(VH_{\bar{f}}1)\otimes(VH_{g}1).
\end{align*}
\end{lemma}
\begin{proof}
By the following identity:
\begin{align*}
I-T_{z}T_{\bar{z}}=1 \otimes 1,
\end{align*}
we have
\begin{align*}
T_{\bar{z}}T_{f}T_{g}T_{z}=&T_{\bar{z}}T_{f}(1 \otimes 1+T_{z}T_{\bar{z}})T_{g}T_{z}\\
=&T_{\bar{z}}T_{f}(1 \otimes 1)T_{g}T_{z}+T_{\bar{z}}T_{f}T_{z}T_{\bar{z}}T_{g}T_{z}\\
=&T_{\bar{z}}T_{f}(1 \otimes 1)T_{g}T_{z}+T_{f}T_{g}\\
=&(T_{\bar{z}f}1)\otimes(T_{\bar{z}\bar{g}}1)+T_{f}T_{g}.
\end{align*}
On the other hand, one easily verifies that
\[T_{\bar{z}f} 1=P \bar{z}f1=PV\bar{f}=V P_{-}\bar{f}=V H_{\bar{f}}1,\]
Thus,
\begin{align*}
T_{\bar{z}}T_{f}T_{g}T_{z}-T_{f}T_{g}
=(VH_{\bar{f}}1)\otimes(VH_{g}1).
\end{align*}
\end{proof}
Next, we present a proof of the result of K. Stroethoff \cite[Theorem 4.4]{Stroethoff1999Algebraic}.
\begin{lemma}\label{1T}
If $f,g,\phi$ and $\psi$ are in $L^{\infty}(\mathbb{T})$, then $T_{f}T_{g}+T_{\phi}T_{\psi}$ is a Toeplitz operator if and only if
\[(VH_{\bar{f}}1)\otimes(VH_{g}1)
+(VH_{\bar{\phi}}1)\otimes(VH_{\psi}1)=0\]
if and only if
one of the following cases holds:
\begin{enumerate}
\item  either $\bar{f}$ or $g$ is analytic and either $\bar{\phi}$ or $\psi$ is analytic;

\item  $f-\lambda\phi \in \overline{H^{2}},\psi+\lambda g\in H^{2},$ where $\lambda \in\mathbb{C}\setminus\{0\}$.

\end{enumerate}
In this case, $T_{f}T_{g}+T_{\phi}T_{\psi}=T_{fg+\phi\psi}.$
\end{lemma}
\begin{proof}
By \cite[Theorem 6]{brown1964algebraic} and Lemma \ref{rank1s} we get that
$T_{f}T_{g}+T_{\phi}T_{\psi}$ is a Toeplitz operator
if and only if
\begin{align*}
T_{\bar{z}}(T_{f}T_{g}+T_{\phi}T_{\psi})T_{z}
=T_{f}T_{g}+T_{\phi}T_{\psi}.
\end{align*}
if and only if
\begin{align}\label{V+V}
(VH_{\bar{f}}1)\otimes(VH_{g}1)+(VH_{\bar{\phi}}1)\otimes(VH_{\psi}1)=0.
\end{align}
If \eqref{V+V} holds,  Lemma \ref{rank20} yields
\begin{enumerate}
\item  either $\bar{f}$ or $g$ is analytic and either $\bar{\phi}$ or $\psi$ is analytic;or

\item   $f-\lambda\phi \in \overline{H^{2}},\psi+\lambda g\in H^{2},$ where $\lambda $ is a constant.
\end{enumerate}

Conversely,  if either $\bar{f}$ or $g$ is analytic and either $\bar{\phi}$ or $\psi$ is analytic, by \cite[Theorem 8]{brown1964algebraic}, we have
\begin{align*}
T_{f}T_{g}+T_{\phi}T_{\psi}=T_{fg+\phi\psi}.
\end{align*}
An easy computation gives
\begin{equation}\label{fg}
\begin{split}
T_{f}T_{g}+T_{\phi}T_{\psi}&=T_{f}T_{g}-T_{fg}+T_{fg}
+T_{\phi\psi}-T_{\phi\psi}+T_{\phi}T_{\psi}\\
&=-H^{*}_{\bar{f}}H_{g}-H^{*}_{\bar{\phi}}H_{\psi}+T_{fg+\phi\psi}
\end{split}
\end{equation}
If $f-\lambda\phi \in \overline{H^{2}},\psi+\lambda g\in H^{2},$ where $\lambda $ is a constant,
then
\begin{align*}
&-H^{*}_{\bar{f}}H_{g}-H^{*}_{\bar{\phi}}H_{\psi}\\
=&-H^{*}_{\bar{\lambda}\bar{\phi}}H_{g}-H^{*}_{\overline{\phi}}H_{-\lambda g}\\
=&-\lambda H^{*}_{\bar{\phi}}H_{g}+\lambda H^{*}_{\bar{\phi}}H_{g}=0.
\end{align*}
\end{proof}
\begin{lemma}\label{2Ts}
If $f,g,\phi$ and $\psi$ are in $L^{\infty},$ then  $T_{f}T_{g}+T_{\phi}T_{\psi}$ is not a Toeplitz operator and is self-adjoint
if and only if one of the following cases holds:
\begin{enumerate}[(1)]
\item either $\bar{f}$ or $g\in H^{2},$
$\bar{\phi}\notin H^{2}$ and $\psi\notin H^{2},$ $\bar{\phi}-a\psi\in H^{2}, a\in\mathbb{R}\setminus \{0\},fg+\phi\psi$ is real-valued.
\item  either $\bar{\phi}$ or $\psi\in H^{2},$
$\bar{f}\notin H^{2}$ and $g\notin H^{2},$
$\bar{f}-bg\in H^{2}, b\in\mathbb{R}\setminus \{0\},fg+\phi\psi$ is real-valued.
\item  $\bar{f},g,\bar{\phi}$ and $\psi$ are not in $H^{2},fg+\phi\psi$ is real-valued.
\begin{enumerate}[(a)]
\item $\bar{f}-\lambda g\in H^{2}$ and $\bar{\phi}-\mu\psi\in H^{2}.$  where $\lambda,\mu\in\mathbb{C}\setminus\{0\};$
\begin{enumerate}
\item $Im(\lambda)=Im(\mu)=0;$
\item $Im(\lambda)\neq 0$ and $Im(\mu)\neq 0,$ $\psi+cg\in H^{2},c\in \mathbb{C}\setminus\{0\},|c|^{2}\frac{Im(\mu)}{Im(\lambda)}=-1.$
\end{enumerate}
\item
$\bar{\phi}-\bar{a}_{11} \bar{f}-\bar{a}_{12}g\in H^{2},$ and $\psi-\bar{a}_{21} \bar{f}-\bar{a}_{22}g\in H^{2},$
where $a_{11},a_{12},a_{21}$ and $a_{22}$ are constant,
$a_{11}\bar{a}_{21}$ and $\bar{a}_{12}a_{22}$ are real numbers,
$\bar{a}_{12}a_{21}-a_{11}\bar{a}_{22}=1.$
\end{enumerate}
\end{enumerate}
\end{lemma}
\begin{proof}
Assume that $T_{f}T_{g}+T_{\phi}T_{\psi}$ is not a Toeplitz operator and  is self-adjoint, we have
\begin{align*}
T_{f}T_{g}+T_{\phi}T_{\psi}&=T_{\bar{g}}T_{\bar{f}}
+T_{\bar{\psi}}T_{\bar{\phi}},\\
T_{\bar{z}}T_{f}T_{g}T_{z}+T_{\bar{z}}T_{\phi}T_{\psi}T_{z}
&=T_{\bar{z}}T_{\bar{g}}T_{\bar{f}}T_{z}+T_{\bar{z}}T_{\bar{\psi}}T_{\bar{\phi}}T_{z}.
\end{align*}
By symbol map \cite[7.11]{douglas2012banach} and \cite[Theorem 2]{Engli1995Toeplitz}, we have
$fg+\phi\psi$ is real-valued.
By Lemma \ref{rank1s} and \cite[Theorem 6]{brown1964algebraic} we get
\begin{align*}
&(VH_{\bar{f}}1)\otimes(VH_{g}1)+(VH_{\bar{\phi}}1)\otimes(VH_{\psi}1)\\
=&(VH_{g}1)\otimes(VH_{\bar{f}}1)+(VH_{\psi}1)\otimes(VH_{\bar{\phi}}1),
\end{align*}
and
\begin{align*}
T_{f}T_{g}+T_{\phi}T_{\psi}\neq
T_{\bar{z}}T_{f}T_{g}T_{z}+T_{\bar{z}}T_{\phi}T_{\psi}T_{z}.
\end{align*}
Hence,
\begin{align*}
(VH_{\bar{f}}1)\otimes(VH_{g}1)+(VH_{\bar{\phi}}1)\otimes(VH_{\psi}1) \end{align*}
is a nonzero self-adjoint operator.

If either $VH_{\bar{f}}1$ or $VH_{g}1$ is the zero vector, and $VH_{\bar{\phi}}1$ and $VH_{\psi}1$ are both nonzero vectors,
then either $\bar{f}\in H^{2}$ or $g\in H^{2},$ and
$\bar{\phi}\notin H^{2}$ and $\psi\notin H^{2}.$
Thus $(VH_{\bar{\phi}}1)\otimes(VH_{\psi}1)$ is a nonzero self-adjoint operator. By Lemma \ref{rank1ss}, we have $\bar{\phi}-a\psi\in H^{2},a\in\mathbb{R}\setminus \{0\}.$

Similarly, if either $VH_{\bar{\phi}}1$ or $VH_{\psi}1$ is the zero vector, and if both $VH_{\bar{f}}1$ and  $VH_{g}1$ are nonzero vectors, then either $\bar{\phi}$ or $\psi\in H^{2},$
$\bar{f}\notin H^{2}$ and $g\notin H^{2},$
$\bar{f}-bg\in H^{2}, b\in\mathbb{R}\setminus \{0\}.$

If $VH_{\bar{f}}1,VH_{g}1,VH_{\bar{\phi}}1,$ and $VH_{\psi}1$ are nonzero vectors, Lemma \ref{rank2s} now gives
\begin{enumerate}[(I)]
\item $\bar{f}-\lambda g\in H^{2}$ and $\bar{\phi}-\mu\psi\in H^{2}.$  where $\lambda,\mu\in\mathbb{C}\setminus\{0\};$
\begin{enumerate}[(i)]
\item $\lambda$ and $\mu$ are real;
\item $Im(\lambda)\neq 0$ and $Im(\mu)\neq 0,$ $\psi+cg\in H^{2},c\in \mathbb{C}\setminus\{0\},|c|^{2}\frac{Im(\mu)}{Im(\lambda)}=-1.$
\end{enumerate}
\item
$\bar{\phi}-\bar{a}_{11} \bar{f}-\bar{a}_{12}g\in H^{2},\quad
\psi-\bar{a}_{21} \bar{f}-\bar{a}_{22}g\in H^{2},$
where $a_{11},a_{12},a_{21}$ and $a_{22}$ are constant,
$a_{11}\bar{a}_{21}$ and $\bar{a}_{12}a_{22}$ are real numbers,
$\bar{a}_{12}a_{21}-a_{11}\bar{a}_{22}=1.$
\end{enumerate}

To verify condition (1), an easy computation gives
\begin{align*}
T_{f}T_{g}+T_{\phi}T_{\psi}
&=T_{fg+\phi\psi}-T_{\phi\psi}+T_{\phi}T_{\psi}\\
&=T_{fg+\phi\psi}-H^{*}_{\bar{\phi}}H_{\psi}\\
&=T_{fg+\phi\psi}-aH^{*}_{\psi}H_{\psi},
\end{align*}
$T_{fg+\phi\psi}-aH^{*}_{\psi}H_{\psi}$ is self-adjoint, and  condition (1) is verified.

Condition (2) is verified in the same way as  condition (1).

To verify condition (3)(a)(i), using \eqref{fg} we obtain
\begin{align*}
T_{f}T_{g}+T_{\phi}T_{\psi}
&=T_{fg+\phi\psi}-H^{*}_{\bar{f}}H_{g}-H^{*}_{\bar{\phi}}H_{\psi}\\
&=T_{fg+\phi\psi}-\lambda H^{*}_{g}H_{g}-\mu H^{*}_{\psi}H_{\psi}.
\end{align*}
therefore, $T_{fg+\phi\psi}-\lambda H^{*}_{g}H_{g}-\mu H^{*}_{\psi}H_{\psi}$ is self-adjoint, and  Condition (3)(a)(i) is verified.

To verify condition (3)(a)(ii): $\bar{f},g,\bar{\phi}$ and $\psi$ are not in $H^{2},fg+\phi\psi$ is real-valued. $\bar{f}-\lambda g\in H^{2}$ and $\bar{\phi}-\mu\psi\in H^{2}.$  where $\lambda,\mu\in\mathbb{C}\setminus\{0\};$
$Im(\lambda)\neq 0$ and $Im(\mu)\neq 0,$ $\psi+cg\in H^{2},c\in \mathbb{C}\setminus\{0\},|c|^{2}\frac{Im(\mu)}{Im(\lambda)}=-1.$
Again using \eqref{fg} we obtain
\begin{align*}
T_{f}T_{g}+T_{\phi}T_{\psi}
&=T_{fg+\phi\psi}-H^{*}_{\bar{f}}H_{g}-H^{*}_{\bar{\phi}}H_{\psi}\\
&=T_{fg+\phi\psi}-\bar{\lambda} H^{*}_{g}H_{g}-\bar{\mu} H^{*}_{\psi}H_{\psi}\\
&=T_{fg+\phi\psi}-\bar{\lambda} H^{*}_{g}H_{g}-\bar{\mu}|c|^{2} H^{*}_{g}H_{g}\\
&=T_{fg+\phi\psi}-(\bar{\lambda}+\bar{\mu}|c|^{2}) H^{*}_{g}H_{g}.
\end{align*}
Since $|c|^{2}\frac{Im(\mu)}{Im(\lambda)}=-1, \bar{\lambda}+\bar{\mu}|c|^{2}$ is a real constant,
$T_{fg+\phi\psi}-(\bar{\lambda}+\bar{\mu}|c|^{2}) H^{*}_{g}H_{g}$ is self-adjoint, Condition (3)(a)(ii) is verified.

To verify condition (3)(b): $\bar{\phi}-\bar{a}_{11} \bar{f}-\bar{a}_{12}g $ and $\psi-\bar{a}_{21} \bar{f}-\bar{a}_{22}g$ are in $H^{2},$
where $a_{11},a_{12},a_{21}$ and $a_{22}$ are constant,
$a_{11}\bar{a}_{21}$ and $\bar{a}_{12}a_{22}$ are real numbers,
$\bar{a}_{12}a_{21}-a_{11}\bar{a}_{22}=1.$ Again using \eqref{fg} we obtain
\begin{align*}
&T_{f}T_{g}+T_{\phi}T_{\psi}\\
=&T_{fg+\phi\psi}-H^{*}_{\bar{f}}H_{g}-H^{*}_{\bar{\phi}}H_{\psi}\\
=&T_{fg+\phi\psi}-H^{*}_{\bar{f}}H_{g}
-(a_{11}H^{*}_{\bar{f}}+a_{12}H^{*}_{g})(\bar{a}_{21}H_{\bar{f}}+\bar{a}_{22}H_{g})\\
=&T_{fg+\phi\psi}-a_{11}\bar{a}_{21}H^{*}_{\bar{f}}H_{\bar{f}}
-a_{12}\bar{a}_{22}H^{*}_{g}H_{g}
-(1+a_{11}\bar{a}_{22})H^{*}_{\bar{f}}H_{g}
-a_{12}\bar{a}_{21}H^{*}_{g}H_{\bar{f}}\\
=&T_{fg+\phi\psi}-a_{11}\bar{a}_{21}H^{*}_{\bar{f}}H_{\bar{f}}
-a_{12}\bar{a}_{22}H^{*}_{g}H_{g}
-\bar{a}_{12}a_{21}H^{*}_{\bar{f}}H_{g}
-a_{12}\bar{a}_{21}H^{*}_{g}H_{\bar{f}}.
\end{align*}
$T_{fg+\phi\psi}-a_{11}\bar{a}_{21}H^{*}_{\bar{f}}H_{\bar{f}}
-a_{12}\bar{a}_{22}H^{*}_{g}H_{g}
-\bar{a}_{12}a_{21}H^{*}_{\bar{f}}H_{g}
-a_{12}\bar{a}_{21}H^{*}_{g}H_{\bar{f}}$ is self-adjoint, Condition (3)(b) is verified.
\end{proof}

\section{The product of two Toeplitz operators is a projection}

The following Lemma is well known. (see\cite[Corollary 1.9,Theorem 2.3,Theorem 2.4]{Peller2003Hankel})
\begin{lemma}\label{lemma1}
Let $R$ be a Hankel operator on $H^{2}.$
\begin{enumerate}
\item $\ker R$ is an
invariant subspace of $T_{z}$;

\item $R$ has nontrivial kernel if and only if the symbol of $R$  has the form $\bar{\theta}\phi$ where $\theta$ is some inner function and $\phi\in H^{\infty}.$
Furthermore:
\begin{enumerate}
\item $\ker H_{\bar{\theta} \phi}=\theta H^{2}, \quad \ker H_{\bar{\theta} \phi}^{*}=\overline{z \theta H^{2}};$
\item  closure $\left\{Range(H_{\bar{\theta} \phi}^{*})\right\}=(\ker H_{\bar{\theta} \phi})^{\perp}=H^{2} \ominus \theta H^{2}=K^{2}_{\theta};$
\item  closure $\left\{Range (H_{\bar{\theta}{\phi}})\right\}=\overline{z K^{2}_{\theta}}.$
\end{enumerate}
\end{enumerate}
\end{lemma}

\begin{lemma}\label{1}
Let $f,g\in L^{\infty}(\mathbb{T})$. If $T_{f}T_{g}$ is a
nontrivial idempotent, then $fg=1$ a.e. on $\mathbb{T}.$
\end{lemma}
\begin{proof}
Suppose $T_{f}T_{g}$ is a nontrivial idempotent, namely,
$\left(T_{f}T_{g}\right)^{2}=T_{f} T_{g}.$
By symbol map \cite[Theorem 7.11]{douglas2012banach}, we have $(fg)^{2}=fg.$ Then there exists a measurable subset $E$ of $\mathbb{T}$ such that
\begin{align*}
(f g)(e^{i \theta})=\left\{\begin{array}{ll}{1}, & {e^{i \theta} \notin E}, \\ {0}, & {e^{i \theta} \in E}.\end{array}\right.
\end{align*}

If $m(E)> 0,$ then there exists a subset $E_{1}$ of $E$ with positive measure, such that either $f|_{E_{1}}=0$ or $g|_{E_{1}}=0.$

If $f|_{E_{1}}=0,$ by Guo' Lemma \cite[Lemma 1]{guo1996problem}, then $\ker{T_{f}}=\ker{T_{\bar{f}}}=\{0\}.$
Since $T_{f}T_{g}$ is a nontrivial idempotent,
$\ker{T_{f}T_{g}}\neq\{0\}.$ For any nonzero vector $x\in\ker{T_{f}T_{g}},$ we have $T_{g}x=0,$ hence $\ker{T_{g}}\neq\{0\}.$ By Coburn' Lemma \cite[7.24]{douglas2012banach},
$\ker{T^{*}_{g}}=\ker{T_{\bar{g}}}=\{0\}.$
Since $T_{f}T_{g}$ is a nontrivial idempotent, $(T_{f}T_{g})^{*}=T_{\bar{g}}T_{\bar{f}}$
is also a nontrivial idempotent.
Hence
$\ker{T_{\bar{g}}T_{\bar{f}}}\neq\{0\},$ it is a contradiction.

If $g|_{E_{1}}=0,$ same considerations apply to $T_{\bar{g}}T_{\bar{f}},$ we can also get a contradiction.
Hence $m(E)=0.$
\end{proof}

\begin{lemma}\label{01}
If a Toeplitz operator is a projection, it must be 0 and $1 .$
\end{lemma}
\begin{proof}
By \cite[Corollary 5]{brown1964algebraic}, the only idempotent Toeplitz operators are 0 and $1 .$
\end{proof}
Hence, a Toeplitz operator cannot be a nontrivial projection.
\begin{theorem}\label{main1}
If $f,g\in L^{\infty}(\mathbb{T}),$ then the following statements
are equivalent.
\begin{enumerate}
\item $T_{f}T_{g}$ is a
nontrivial projection;

\item $T_{f}T_{g}$ is a
projection, and its range is a nontrivial invariant subspace of the shift operator $T_{z};$

\item There exist a nonconstant inner function $\theta$ and a nonzero constant $a$ such that $f=a\theta$ and $ g=\frac{\bar{\theta}}{a}$.

\end{enumerate}
\end{theorem}
\begin{proof}
$(1)\Rightarrow(2):$
Suppose $T_{f}T_{g}$ is a
nontrivial projection,
by Lemma \ref{2Ts} (2), we have  $\bar{f}=\lambda g+h,\lambda\in \mathbb{R}\setminus \{0\}, h\in H^{2}.$
$T_{f}T_{g}$ is a
nontrivial projection  if and only if
$I-T_{f}T_{g}$  is a nontrivial projection.
By Lemma \ref{1}, we have $I=T_{fg}.$ Hence
\begin{align*}
I-T_{f} T_{g}
&=T_{fg}-T_{f} T_{g} \\
&=H^{*}_{\bar{f}}H_g\\
&=\left(\lambda H_{g}^{*}+H_{h}\right) H_{g}\\
&=\lambda H_{g}^{*}H_{g},
\end{align*}
thus $\lambda H_{g}^{*}H_{g}$ is a nontrivial projection.

By Lemma \ref{lemma1} (1), $\ker H_{g}$ is an invariant subspace of shift operator $T_{z}.$
Moreover,
\begin{align*}
\ker H_{g}=\ker H_{g}^{*}H_{g}=\ker\lambda H_{g}^{*}H_{g}=\ker(I-T_{f} T_{g})=Range(T_{f} T_{g}).
\end{align*}
Therefore, the range of $T_{f}T_{g}$ is a nontrivial invariant subspace of the shift operator $T_{z}.$

$(2)\Rightarrow(3):$
By Beurling's theorem \cite[6.11]{douglas2012banach}, $Range(T_{f} T_{g})=\theta H^{2}$ for some nonconstant inner function $\theta.$
$T_{\theta}T_{\bar{\theta}}$ is the orthogonal projection of $L^{2}$ onto $\theta H^{2}.$ Hence
\begin{align*}
T_{f}T_{g}=T_{\theta}T_{\bar{\theta}}.
\end{align*}
By Lemma \ref{1T},
we have
\begin{align*}
f-a\theta\in \overline{H^{2}}
,\bar{\theta}-ag\in H^{2},a \in\mathbb{C}\setminus\{0\}.
\end{align*}
Note that
\begin{align*}
T_{\frac{1}{a}f}T_{ag}=T_{f}T_{g},
\end{align*}
let
\begin{equation}{\label{u}}
\begin{split}
F\triangleq\frac{1}{a}f&=\theta+\bar{\varphi},\\
G\triangleq ag&=\psi+\bar{\theta},
\end{split}
\end{equation}
where $\varphi$ and $\psi$ are in $H^{\infty}.$
Since Lemma \ref{1}, $FG=1,$  $\bar{\theta}F\theta G=1$ and $\theta\bar{F}\theta G$ is an inner function.

If $\theta\bar{F}\theta G\neq1,$
then $Re(1-\theta \bar{F}\theta G)>0,$ by \cite[Part A. 4.2.2]{nikolski2009operators}, we have $1-\theta \bar{F}\theta G$ is outer. Using \eqref{u}, then
\begin{align*}
1-\theta \bar{F}\theta G
&=1-\theta (\overline{\theta }+\varphi)\theta (\overline{\theta }+\psi)\\
&=1-(1+\theta \varphi)(1+\theta \psi)\\
&=-\theta(\varphi+\psi+\varphi \psi),
\end{align*}
it is a contradiction. Hence
$\theta \bar{F}\theta G=1.$
Note that $\theta \bar{F}$ and $\theta G$ are in $H^{\infty},$
by \cite[6.20]{douglas2012banach}, $\theta \bar{F}$ and $\theta G$ are outer functions.
Since $\theta\bar{F}\theta G=1=\bar{\theta}F\theta G,$
$\theta\bar{F}=\frac{1}{\theta G},$
$\theta\bar{F}=\overline{(\frac{1}{\theta G})},$
and
$\theta\bar{F}$ and $\theta G$ are real-valued functions in $H^{\infty},$ there exists  a nonzero real constant $c$ such that
\begin{align*}
F=c\theta\quad\text{and}\quad
G=\frac{\bar{\theta}}{c}.
\end{align*}
Combining this with \eqref{u}, we arrive at
\begin{align*}
(c-1)\theta=\bar{\varphi}\quad\text{and}\quad
(\frac{1}{c}-1)\bar{\theta}=\psi.
\end{align*}
Since $\theta$ is not a constant, $c=1,$ it follows that
\begin{align*}
F=\theta\quad\text{and}\quad
G=\bar{\theta}.
\end{align*}
From \eqref{u}, we have
\begin{align*}
f=a\theta\quad\text{and}\quad
g=\frac{1}{a}\bar{\theta}.
\end{align*}

$(3)\Rightarrow(1):$
Suppose $f=a\theta$ and $ g=\frac{\bar{\theta}}{a}$.Then
\begin{align*}
T_{f} T_{g}
=T_{\theta}T_{\bar{\theta}}.
\end{align*}
Hence $T_{f}T_{g}$ is a
nontrivial projection operator.
\end{proof}

\begin{remark}
Widom \cite[7.46]{douglas2012banach} proved that the spectrum of a Toeplitz operator is a connected subset of complex plane. It is natural to ask  whether the spectrum of the product of two Toeplitz operator is connected? Since the spectrum of a projection operator is $\{0,1\},$ by Theorem \ref{01},  the answer to the question is negative.
\end{remark}
\begin{lemma}\cite[Theorem 7.22]{fricain2016theory}\label{partial}
Let $\mathcal{H}_{1}$ and $\mathcal{H}_{2}$ be Hilbert spaces, and $A$ an operator in $\mathcal{L}\left(\mathcal{H}_{1}, \mathcal{H}_{2}\right)$ Then the following are equivalent:
\begin{enumerate}
\item $A$ is a partial isometry;
\item $A^{*}$ is a partial isometry;
\item $A A^{*}$ is an orthogonal projection, $A A^{*}=P_{(\ker A^{*})^{\bot}}$;
\item $A^{*} A$ is an orthogonal projection, $A^{*} A=P_{(\ker A)^{\bot}}$.
\end{enumerate}
\end{lemma}
Using  Theorem \ref{main1}, we present
a new proof of the result of
A. Brown and R. Douglas \cite{brown1965partially}.
\begin{corollary}
If $f\in L^{\infty}$ then the following statements
are equivalent.
\begin{enumerate}

\item $T_{f}$ is a
partial isometry;

\item $T^{*}_{f}$ is a
partial isometry;

\item either $f$ or $\bar{f}$ is inner.
\end{enumerate}
\end{corollary}
\begin{proof}
Using  Lemma \ref{partial} and Theorem \ref{main1}.
\end{proof}
\section{The product of two Hankel operators is a projection}
\begin{theorem}\label{2H}
If $f,g\in L^{\infty}$ then the following statements
are equivalent.
\begin{enumerate}

\item $H^{*}_{\bar{f}}H_{g}$ is a
nontrivial projection operator;

\item The range of $H^{*}_{\bar{f}}H_{g}$ is
a model space  $K^{2}_{\theta},$ where $\theta$ is an inner function;

\item$\bar{f}+\bar{\mu}\bar{\theta},g+\frac{\bar{\theta}}{\mu}\in H^{2},$ where $\mu \in\mathbb{C}\setminus\{0\}$.

\end{enumerate}
\end{theorem}
\begin{proof}
$(1)\Rightarrow(2):$
We can suppose  $T_{fg}-T_{f}T_{g}$ is a
nontrivial projection because
$H^{*}_{\bar{f}}H_{g}=T_{fg}-T_{f}T_{g}.$
By Lemma \ref{2Ts}(2), we have
\begin{align}\label{f}
\bar{f}=\lambda g+h,\quad \lambda\in \mathbb{R}\setminus \{0\},\quad h\in H^{2}.
\end{align}

Moreover,
\begin{align*}
T_{fg}-T_{f} T_{g}
&=H^{*}_{\bar{f}}H_g\\
&=\left(\lambda H_{g}^{*}+H_{h}\right) H_{g}\\
&=\lambda H_{g}^{*}H_{g},
\end{align*}
Hence,
\begin{align*}
\ker (H^{*}_{\bar{f}}H_g)=\ker(\lambda H_{g}^{*}H_{g})=\ker H_{g}.
\end{align*}
By Lemma \ref{lemma1} (1), $\ker H_{g}$ is an invariant subspace of shift operator $T_{z},$  by Beurling's theorem \cite[6.11]{douglas2012banach}, $\ker (H^{*}_{\bar{f}}H_g)=\theta H^{2}$ for some nonconstant inner function $\theta.$
$T_{\theta}T_{\bar{\theta}}$ is the orthogonal projection of $L^{2}$ onto $\theta H^{2}.$ Hence
\begin{align*}
\lambda H_{g}^{*}H_{g}=&I-T_{\theta}T_{\bar{\theta}}\\
\lambda(T_{\bar{g}g}-T_{\bar{g}}T_{g})=&I-T_{\theta}T_{\bar{\theta}}\\
T_{\lambda |g|^{2}-1}=&T_{\lambda\bar{g}}T_{g}
-T_{\theta}T_{\bar{\theta}}
\end{align*}
Since projection operator is positive, $\lambda>0.$
By Lemma \ref{1T}, we have
\begin{align}\label{lambda}
\lambda \bar{g}+\mu\theta\in\overline{ H^{2}},\quad
\bar{\theta}+\mu g\in H^{2},\quad \mu \in\mathbb{C}\setminus\{0\}.
\end{align}
Hence,
\begin{align*}
(\lambda-|\mu|^{2})g\in H^{2},
\end{align*}
If $\lambda\neq|\mu|^{2},$  then
$g\in H^{2}$ and $H_{g}=0.$ By assumption $T_{fg}-T_{f} T_{g}=\lambda H_{g}^{*}H_{g}$ is a nontrivial projection, so
$\lambda=|\mu|^{2}.$
Using \eqref{lambda}, we have
\begin{align}\label{g}
g+\frac{\bar{\theta}}{\mu}\in H^{2}.
\end{align}
Combining \eqref{g} with \eqref{f}  gives
\begin{align*}
f+\mu\theta\in \overline{H^{2}}.
\end{align*}
$(4)\Rightarrow(1):$
Suppose $f+\mu\theta\in \overline{H^{2}}, g+\frac{\bar{\theta}}{\mu}\in H^{2},\mu \in\mathbb{C}\setminus\{0\}.$Then
\begin{align*}
H^{*}_{\bar{f}}H_g
&=H^{*}_{\bar{\mu}\bar{\theta}}H_{\frac{\bar{\theta}}{\mu}}\\
&=H^{*}_{\bar{\theta}}H_{\bar{\theta}}\\
&=I-T_{\theta}T_{\bar{\theta}}.
\end{align*}
$I-T_{\theta}T_{\bar{\theta}}$ is the projection onto $K^{2}_{\theta}.$
\end{proof}
Next, we derive an alternative proof of \cite[Theorem 2.6]{Peller2003Hankel}.
\begin{corollary}
If $f\in L^{\infty}$
then  $H_{f}$ is a
partial isometry if and only if
$\bar{f}$ is inner.
\end{corollary}
\begin{proof}
Using  Lemma \ref{partial} and Theorem \ref{2H}.
\end{proof}

\section{Projection as self-commutators of Toeplitz operators}

The problem 237 in Paul R.Halmos's famous text: A Hilbert space problem book \cite{Halmos1978A} states: can $T^{*}T-TT^{*}$ be a projection and, if so, how? He discuss the following two cases.

(a) If $T$ is an abnormal operator of norm 1, such that $T^{*}T-TT^{*}$ is a projection, then $T$ is an isometry.

(b) Does the statement remain true if the norm condition is not assumed?

In particular, if $T$ is a Toeplitz operator.
Let $f\in L^{\infty}(\mathbb{T}),$
we consider that
when is $T^{*}_{f}T_{f}-T_{f}T^{*}_{f}$  a projection?

Define
\begin{align*}
Q=T^{*}_{f}T_{f}-T_{f}T^{*}_{f}.
\end{align*}
\begin{example}\label{E1}
Corresponding case (a), we next show that if there is a constant $\lambda$ such that $\|T_{f+\lambda}\|\leq1$ and $Q$ is a nontrivial projection, then $T_{f+\lambda}$ is an isometry.

Note that \[T^{*}_{f+\lambda}T_{f+\lambda}-T_{f+\lambda}T^{*}_{f+\lambda}
=T^{*}_{f}T_{f}-T_{f}T^{*}_{f},\quad \lambda\in\mathbb{C}.\]

Using the idea of \cite[Solution 237]{Halmos1978A} and $\|T_{f}\|=\|f\|_{\infty}$ we have
\begin{align*}
\|h\|^{2}\geq\|T_{f+\lambda}h\|^{2}=\langle T^{*}_{f+\lambda}T_{f+\lambda}h,h\rangle
=&\langle T_{f+\lambda}T^{*}_{f+\lambda}h,h\rangle+\langle Qh,h\rangle\\
=&\|T^{*}_{f+\lambda}h\|^{2}+\|Qh\|^{2}.
\end{align*}
Replace $h$ by $Qx(x\in H^{2})$ in the above formula, we have
$T^{*}_{f+\lambda}Q=0$ and $T_{f+\lambda}$ is quasinormal.
A Theorem in \cite{amemiya1975quasinormal} tells us that
a quasinormal Toeplitz operator is either normal or analytic and $f+\lambda=c\theta,$ where $c$ is a constant and $\theta$ is an inner function.
$Q$ is a nontrivial projection, we have $f+\lambda=c\theta.$
Hence,
\begin{align*}
Q=&T^{*}_{f}T_{f}-T_{f}T^{*}_{f}\\
=&T_{|f|^{2}}-T_{f}T_{\bar{f}}\\
=&H^{*}_{\bar{f}}H_{\bar{f}}\\
=&|c|^{2}H^{*}_{\bar{\theta}}H_{\bar{\theta}}.
\end{align*}
Since $Q$ is an idempotent, $|c|=1.$
By \cite[Corollary 3]{brown1964algebraic},  $T_{f+\lambda}$ is an isometry if and only if $f+\lambda$ is an inner function.
In this case, $Q=H^{*}_{\bar{\theta}}H_{\bar{\theta}}$ is the projection onto model space $K^{2}_{\theta}.$
\end{example}

\begin{example}
Let us recall Abrahamse's theorem \cite{abrahamse1976subnormal}. If
\begin{enumerate}
\item $f$ or $\bar{f}$ is of bounded type;

\item $T_{f}$ is hyponormal;

\item $\ker Q$ is invariant for $T_{f}$.
\end{enumerate}
then $T_{f}$ is normal or analytic.

Using the above theorem,  if
\begin{enumerate}
\item $f$ or $\bar{f}$ is of bounded type;

\item $Q$ is a nontrivial projection;

\item $\ker Q$ is invariant for $T_{f},$
\end{enumerate}
then $f$ is analytic. Hence,
\begin{align*}
Q=&T^{*}_{f}T_{f}-T_{f}T^{*}_{f}\\
=&T_{|f|^{2}}-T_{f}T_{\bar{f}}\\
=&H^{*}_{\bar{f}}H_{\bar{f}}.
\end{align*}
By Theorem \ref{2H}, there is a constant $c$ such that
$f=\theta+c,$ where $\theta$ is an inner function.
In this case, $Q=H^{*}_{\bar{\theta}}H_{\bar{\theta}}$ is the projection onto model space $K^{2}_{\theta}.$
\end{example}

From the above two examples, we need to consider two things:
if $Q$ is a nontrivial projection,

1. when is  the range of $Q$ a model space?

2. is the range of $Q$ necessarily a model space?

\begin{lemma}\label{K2}
If $\varphi\in L^{\infty}$  then
$T^{*}_{\varphi}T_{\varphi}-T_{\varphi}T^{*}_{\varphi}$ is the projection  on to a model space $K^{2}_{\theta}$
if and only if
$\varphi=a\theta+b\bar{\theta}+c,$ where $a,b$ and $c$ are constant with $|a|^{2}-|b|^{2}=1.$
\end{lemma}
\begin{proof}
If $\varphi=a\theta+b\bar{\theta}+c,$ where $a,b$ and $c$ are constant with $|a|^{2}-|b|^{2}=1,$ then
\begin{align*}
T^{*}_{\varphi}T_{\varphi}-T_{\varphi}T^{*}_{\varphi}
=&T^{*}_{\varphi}T_{\varphi}-T_{\varphi\bar{\varphi}}
+T_{\varphi\bar{\varphi}}-T_{\varphi}T^{*}_{\varphi}\\
=&H^{*}_{\bar{\varphi}}H_{\bar{\varphi}}-H^{*}_{\varphi}H_{\varphi}\\
=&(|a|^{2}-|b|^{2})H^{*}_{\bar{\theta}}H_{\bar{\theta}}\\
=&H^{*}_{\bar{\theta}}H_{\bar{\theta}}.
\end{align*}

Conversely, suppose $T^{*}_{\varphi}T_{\varphi}-T_{\varphi}T^{*}_{\varphi}$ is the projection  on to a model space $K^{2}_{\theta},$ then
\begin{align*}
T^{*}_{\varphi}T_{\varphi}-T_{\varphi}T^{*}_{\varphi}
=I-T_{\theta}T_{\bar{\theta}}.
\end{align*}
Write $\varphi=f+\bar{g},$ $f$ and $g$ in $H^{2},$ using Lemma \ref{rank1s}, we have
\begin{align}\label{VVV}
&(VH_{\bar{g}}1)\otimes(VH_{\bar{g}}1)
-(VH_{\bar{f}}1)\otimes(VH_{\bar{f}}1)
=-(VH_{\bar{\theta}}1)\otimes(VH_{\bar{\theta}}1).
\end{align}
\noindent{\bf{Case 1.}}\,\,\,

Assume that   $\left\{H_{\bar{g}}1, H_{\bar{f}} 1\right\}$
is linearly dependent, there are two constants $k_{1}$ and $k_{2}$ such that
\begin{align*}
k_{1} H_{\bar{g}}1+k_{2}H_{\bar{f}} 1=0.
\end{align*}
If $k_{1}$ is not zero, let $\lambda=-\frac{k_{2}}{k_{1}},$
then
$\bar{g}-\lambda\bar{f}\in H^{2}$ and
\begin{align*}
T^{*}_{\varphi}T_{\varphi}-T_{\varphi}T^{*}_{\varphi}
=&H^{*}_{\bar{\varphi}}H_{\bar{\varphi}}-H^{*}_{\varphi}H_{\varphi}\\
=&H^{*}_{\bar{f}}H_{\bar{f}}-H^{*}_{\bar{g}}H_{\bar{g}}\\
=&(1-|\lambda|^{2})H^{*}_{\bar{f}}H_{\bar{f}}.
\end{align*}
Then $(1-|\lambda|^{2})H^{*}_{\bar{f}}H_{\bar{f}}
=H^{*}_{\bar{\theta}}H_{\bar{\theta}}$ is a projection, and
$1-|\lambda|^{2}>0.$ By Theorem \ref{2H},
we have $f+\frac{\mu}{\sqrt{1-|\lambda|^{2}}}\theta \in \overline{H^{2}}, \mu$ is unimodular constant.
Therefore,
$\varphi=-\frac{\mu}{\sqrt{1-|\lambda|^{2}}}\theta
-\frac{\lambda\mu}{\sqrt{1-|\lambda|^{2}}}\bar{\theta}+c,$
where $c$ is a constant.
Let $a=-\frac{\mu}{\sqrt{1-|\lambda|^{2}}}$ and
$b=-\frac{\lambda\mu}{\sqrt{1-|\lambda|^{2}}},$ we have
\begin{align}\label{varphi}
\varphi=a\theta+b\bar{\theta}+c,
\end{align}
where $|a|^{2}-|b|^{2}=1.$

If $k_{2}$ is not zero, repeating the previous reasoning, we can prove the same equality \eqref{varphi} hold.

\noindent{\bf{Case 2.}}\,\,\,

Assume that $\left\{H_{\bar{g}}1, H_{\bar{f}} 1\right\}$
is linearly independent.
Since $V$ is anti-unitary, $\left\{VH_{\bar{g}} 1, VH_{\bar{f}} 1\right\}$
is linearly independent, by Gram-Schmidt procedure, there
exist a nonzero function $x_{0}$ in
$span\left\{VH_{\bar{g}} 1, VH_{\bar{f}} 1\right\}$ such that
\begin{align*}
\langle VH_{\bar{g}}1,x_{0}\rangle&=1,\\
\langle VH_{\bar{f}}1,x_{0}\rangle&=0.
\end{align*}
Applying operator equation \eqref{VVV} to $x_{0}$ gives
\begin{align*}
VH_{\bar{g}}1=&-\langle x_{0},VH_{\bar{\theta}}1\rangle VH_{\bar{\theta}}1,\\
H_{\bar{g}}1=&-\langle VH_{\bar{\theta}}1,x_{0}\rangle H_{\bar{\theta}}1.
\end{align*}
Let $b=-\langle x_{0},VH_{\bar{\theta}}1\rangle,$ thus
$g-b\theta\in \overline{H^{2}},$ and $g-b\theta$ is a constant.

Similarly, there exists a constant $a$ such
$f-a\theta$ is a constant. Therefore,
\begin{align*}
T^{*}_{\varphi}T_{\varphi}-T_{\varphi}T^{*}_{\varphi}
=&H^{*}_{\bar{f}}H_{\bar{f}}-H^{*}_{\bar{g}}H_{\bar{g}}\\
=&(|a|^{2}-|b|^{2})H^{*}_{\bar{\theta}}H_{\bar{\theta}}.
\end{align*}
and $|a|^{2}-|b|^{2}=1.$
\end{proof}
Recall the definition of truncated Toeplitz operator.
For $\varphi$ in $L^{2}(\mathbb{T}),$
the truncated Toeplitz operator $A^{\vartheta}_{\varphi}$ is densely defined on $K^{2}_{\vartheta}$ by
\begin{align*}
A^{\vartheta}_{\varphi}f
=(P-T_{\bar{\vartheta}}T_{\vartheta})({\varphi}f).
\end{align*}
The algebraic properties of truncated Toeplitz operator will paly key role in the following Lemma.
\begin{lemma}\label{SC}
If $\varphi\in L^{\infty}$ then
$T^{*}_{\varphi}T_{\varphi}-T_{\varphi}T^{*}_{\varphi}$ is a nontrivial projection operator and its range is not a Model space
if and only if
 $\varphi=uv+\bar{v}+a,$ where $u$ is inner, $v\in H^{2}$ with $|v|^{2}-1\in uH^{2}+\overline{uH^{2}}$ and $a$ is constant.
\end{lemma}

\begin{proof}
Assume that $T^{*}_{\varphi}T_{\varphi}-T_{\varphi}T^{*}_{\varphi}$ is a projection.
Since  projection is positive,
$T^{*}_{\varphi}T_{\varphi}-T_{\varphi}T^{*}_{\varphi}$ is positive and  $T_{\varphi}$ is hyponormal.
We recall the characterization of Hyponormality of Toeplitz operators form Carl C. Conwen \cite{cowen1988hyponormality}. The theorem can be stated as follows:

 If $\varphi$ is in $L^{\infty}(\mathbb{T}),$ where $\varphi=f+\bar{g}$ for $f$ and $g$ in $H^{2},$ then $T_{\varphi}$ is hyponormal if and only if
\begin{align}\label{fg}
g=c+T_{\bar{u}} f
\end{align}
for some constant $c$ and some function $u$ in $H^{\infty}$ with $\|u\|_{\infty} \leq 1.$

According to Conwen's Theorem, if $Q$ is a nontrivial projection, using \eqref{fg}, we have
\begin{align*}
Q=&H^{*}_{\bar{\varphi}}H_{\bar{\varphi}}-H^{*}_{\varphi}H_{\varphi}\\
=&H^{*}_{\bar{f}}H_{\bar{f}}-H^{*}_{\bar{g}}H_{\bar{g}}\\
=&H^{*}_{\bar{f}}H_{\bar{f}}-H^{*}_{\overline{{T_{\bar{u}}f}}}H_{\overline{{T_{\bar{u}}f}}}\\
=&H^{*}_{\bar{f}}H_{\bar{f}}
-H^{*}_{\overline{{P{\bar{u}}f}}+\overline{{P_{-}{\bar{u}}f}}}
H_{\overline{{P{\bar{u}}f}}+\overline{{P_{-}{\bar{u}}f}}}\\
=&H^{*}_{\bar{f}}H_{\bar{f}}-H^{*}_{u\bar{f}}H_{u\bar{f}}\\
=&H^{*}_{\bar{f}}H_{\bar{f}}-H^{*}_{\bar{f}}S_{\bar{u}}S_{u}H_{\bar{f}}\\
=&H^{*}_{\bar{f}}(I-S_{\bar{u}}S_{u})H_{\bar{f}}.
\end{align*}
where $S_{u}x=P_{-}(ux), x\in (H^{2})^{\perp}.$

Since $\|u\|_{\infty}\leq1$ and $\|S_{u}\|=\|u\|_{\infty},S_{u}$
is a contraction, we have $I-S_{\bar{u}}S_{u}$ is positive, and
$\ker(I-S_{\bar{u}}S_{u})=\ker (I-S_{\bar{u}}S_{u})^{1/2}.$

We claim that if $I-S_{\bar{u}}S_{u}$ is not injective, then $u$ is an inner function. To see this,  let $x$ be a nonzero vector such that $(I-S_{\bar{u}}S_{u})x=0.$

Hence,
\begin{align*}
\left\langle (I-S_{\bar{u}}S_{u})x,x\right\rangle
=&\left\langle x,x\right\rangle-\left\langle S_{u}x,S_{u}x\right\rangle\\
=&\|x\|^{2}-\|S_{u}x\|^{2}=0
\end{align*}
and
\begin{align*}
\int_{\mathbb{T}}|x|^{2}dm=\|x\|^{2}=\|S_{u}x\|^{2}=\|P_{-}ux\|^{2}\leq\|ux\|^{2}
=\int_{\mathbb{T}}|ux|^{2}dm.
\end{align*}
Since $\|u\|_{\infty}\leq1,$  \[|ux|^{2}-|x|^{2}=(|u|^{2}-1)|x|^{2}\leq 0.\]
But $\int_{\mathbb{T}}(|u|^{2}-1)|x|^{2}dm\geq0,$ thus $(|u|^{2}-1)|x|^{2}=0.a.e$ on $\mathbb{T}.$ Hence, $|u|=1
.a.e$ on $\mathbb{T},$ and $u$ is an inner function.

Write
\begin{align*}
Q=&H^{*}_{\bar{f}}(I-S_{\bar{u}}S_{u})^{1/2}(I-S_{\bar{u}}S_{u})^{1/2}H_{\bar{f}}\\
=&((I-S_{\bar{u}}S_{u})^{1/2}H_{\bar{f}})^{*}(I-S_{\bar{u}}S_{u})^{1/2}H_{\bar{f}},
\end{align*}
note that $\ker Q=\ker ((I-S_{\bar{u}}S_{u})^{1/2}H_{\bar{f}}).$
According to the  above claim, we have that if  $u$ is not an inner function, then $\ker (I-S_{\bar{u}}S_{u})^{1/2}=\ker(I-S_{\bar{u}}S_{u})=\{0\}$
and $\ker Q=\ker H_{\bar{f}}.$
By Lemma \ref{lemma1}(1), $\ker H_{\bar{f}}$ is an invariant subspace of $T_{z}.$
Hence, the range of $Q$ is a model space, it is a contradiction.

It remains to consider the case that $u$ be an inner function. Write
\begin{align*}
Q=&H^{*}_{\bar{f}}(I-S_{\bar{u}}S_{u})H_{\bar{f}}\\
=&H^{*}_{\bar{f}}(S_{\bar{u}u}-S_{\bar{u}}S_{u})H_{\bar{f}}\\
=&H^{*}_{\bar{f}}H_{\bar{u}}H^{*}_{\bar{u}}H_{\bar{f}}.
\end{align*}
By Gu's theorm \cite[Theorem 1.1]{gu2001separation}, for two Hankel operatrs $H_{\bar{u}}$ and $H_{\bar{f}},$ either $\ker H^{*}_{\bar{u}}H_{\bar{f}}=\ker H_{\bar{f}}$   or
$\ker H^{*}_{\bar{f}}H_{\bar{u}}=\ker H_{\bar{u}}.$

If $\ker H^{*}_{\bar{u}}H_{\bar{f}} =\ker H_{\bar{f}} ,$ then
$\ker Q=\ker H^{*}_{\bar{u}}H_{\bar{f}} =\ker H_{\bar{f}}.$
By Lemma \ref{lemma1}(1), $\ker H_{\bar{f}}$ is an invariant subspace of $T_{z}.$
Hence, the range of $Q$ is a model space, it is a contradiction.

By Lemma \ref{partial},  $H^{*}_{\bar{f}}H_{\bar{u}}H^{*}_{\bar{u}}H_{\bar{f}}$ is an orthogonal projection, then $H^{*}_{\bar{u}}H_{\bar{f}}H^{*}_{\bar{f}}H_{\bar{u}}$ is  an orthogonal projection.

If $\ker H^{*}_{\bar{f}}H_{\bar{u}}
=\ker H_{\bar{u}}=uH^{2}$(Lemma \ref{lemma1}(2)(a)),
then
\begin{align}\label{6H}
H^{*}_{\bar{u}}H_{\bar{f}}H^{*}_{\bar{f}}H_{\bar{u}}f
=H^{*}_{\bar{u}}H_{\bar{u}}.
\end{align}
Using the property $V,$ we have
\begin{align*}
VH^{*}_{\bar{u}}H_{\bar{f}}H^{*}_{\bar{f}}H_{\bar{u}}V
=&H_{\bar{u}}H^{*}_{\bar{f}}H_{\bar{f}}H^{*}_{\bar{u}},\\
VH^{*}_{\bar{u}}H_{\bar{u}}V
=&H_{\bar{u}}H^{*}_{\bar{u}}.
\end{align*}
Hence
\begin{align}\label{H4}
H_{\bar{u}}H^{*}_{\bar{f}}H_{\bar{f}}H^{*}_{\bar{u}}
=H_{\bar{u}}H^{*}_{\bar{u}}.
\end{align}
Note  that  $\ker H^{*}_{\bar{u}}=\overline{zuH^{2}}$(Lemma \ref{lemma1}(2)(a))
and
$\overline{zH^{2}}\ominus\overline{zu H^{2}}=\overline{zK^{2}_{u}}=\bar{u}K^{2}_{u}.$

For every $h\in K^{2}_{u},$ we have
$H^{*}_{\bar{u}}\bar{u}h=P(u\bar{u}h)=h,$ and
\begin{align*}
\langle H_{\bar{u}}H^{*}_{\bar{f}}H_{\bar{f}}H^{*}_{\bar{u}}\bar{u}h,\bar{u}h\rangle
=&\langle H_{\bar{u}}H^{*}_{\bar{u}}\bar{u}h,\bar{u}h\rangle,\\
\langle H^{*}_{\bar{f}}H_{\bar{f}}H^{*}_{\bar{u}}\bar{u}h
,H^{*}_{\bar{u}}\bar{u}h\rangle
=&\langle H^{*}_{\bar{u}}\bar{u}h,H^{*}_{\bar{u}}\bar{u}h\rangle,\\
\langle H^{*}_{\bar{f}}H_{\bar{f}}h,h\rangle
=&\langle h,h\rangle.
\end{align*}
Hence
\begin{align*}
P_{K^{2}_{u}}(H^{*}_{\bar{f}}H_{\bar{f}})|_{K^{2}_{u}}=I_{K^{2}_{u}}.
\end{align*}
where $P_{K^{2}_{u}}$ is the orthogonal projection onto $K^{2}_{u}$ and $I_{K^{2}_{u}}$ is the
identity operator on $K^{2}_{u}.$

An easy computation gives
\begin{align*}
P_{K^{2}_{u}}H^{*}_{\bar{f}}H_{\bar{f}}h
=&P_{K^{2}_{u}}Pf(I-P)\bar{f}h\\
=&P_{K^{2}_{u}}f(I-P)\bar{f}h\\
=&P_{K^{2}_{u}}f\bar{f}h-P_{K^{2}_{u}}fP\bar{f}h\\
=&P_{K^{2}_{u}}f\bar{f}h-P_{K^{2}_{u}}f(P-uP\bar{u}+uP\bar{u})\bar{f}h\\
=&P_{K^{2}_{u}}f\bar{f}h-P_{K^{2}_{u}}f(P_{K^{2}_{u}}+uP\bar{u})\bar{f}h\\
=&P_{K^{2}_{u}}f\bar{f}h-P_{K^{2}_{u}}fP_{K^{2}_{u}}\bar{f}h\\
=&A^{u}_{|f|^{2}}h-A^{u}_{f}A^{u}_{\bar{f}}h.
\end{align*}
Hence
\begin{equation}\label{A}
\begin{split}
A^{u}_{|f|^{2}}-A^{u}_{f}A^{u}_{\bar{f}}=&I_{K^{2}_{u}},\\
A^{u}_{f}A^{u}_{\bar{f}}=&A^{u}_{|f|^{2}-1}.
\end{split}
\end{equation}

Since $f$ is analytic,
using N. A. Sedlock' theorem \cite[Theorem 5.2]{sedlock2010algebras} leads to $A^{u}_{f}=cI_{K^{2}_{u}},$ where $c$ is a constant,  and $A^{u}_{f-c}$ is the zero operator, then $f-c\in uH^{2}$\cite[Theorem 3.1]{Sarason2007Algebraic}. There is a function $v\in H^{2},$ such that $f=c+uv.$
Since \eqref{fg}, $\varphi=uv+\bar{v}+a,$ where $a$ is a constant.

Substituting $f=c+uv$ into \eqref{6H}, we have
\begin{align}
H^{*}_{\bar{u}}H_{\bar{u}\bar{v}}H^{*}_{\bar{u}\bar{v}}H_{\bar{u}}
=H^{*}_{\bar{u}}H_{\bar{u}}.
\end{align}
Repeating the above reasoning form \eqref{H4} to \eqref{A}, we obtain
\begin{align*}
A^{u}_{uv}A^{u}_{\bar{v}\bar{u}}=A^{u}_{|v|^{2}-1}.
\end{align*}
Note that $A^{u}_{uv}=A^{u}_{f-c}=0,$
hence $A^{u}_{|v|^{2}-1}$ is zero operator,
using \cite[Theorem 3.1]{Sarason2007Algebraic} again,
we have $|v|^{2}-1\in uH^{2}+\overline{uH^{2}}.$

Conversely, if
$\varphi=uv+\bar{v}+c,$ where $u$ is inner, $v\in H^{2}$ with $|v|^{2}-1\in uH^{2}+\overline{uH^{2}}$ and $c$ is constant,
by Lemma \ref{K2}, the range of $Q$ is not a model space.
An easy computation gives
\begin{equation}\label{QQ}
\begin{split}
T^{*}_{\varphi}T_{\varphi}-T_{\varphi}T^{*}_{\varphi}
=&H^{*}_{\bar{\varphi}}H_{\bar{\varphi}}-H^{*}_{\varphi}H_{\varphi}\\
=&H^{*}_{\bar{u}\bar{v}}H_{\bar{u}\bar{v}}
-H^{*}_{\bar{v}}H_{\bar{v}}\\
=&T_{uv\bar{u}\bar{v}}-T_{uv}T_{\bar{u}\bar{v}}
-(T_{v\bar{v}}-T_{v}T_{\bar{v}})\\
=&T_{|v|^{2}}-T_{uv}T_{\bar{u}\bar{v}}
-(T_{|v|^{2}}-T_{v}T_{\bar{v}})\\
=&T_{v}T_{\bar{v}}-T_{uv}T_{\bar{u}\bar{v}}\\
=&T_{v}T_{\bar{v}}-T_{v}T_{u}T_{\bar{u}}T_{\bar{v}}\\
=&T_{v}(I-T_{u}T_{\bar{u}})T_{\bar{v}}\\
=&T_{v}H^{*}_{\bar{u}}H_{\bar{u}}T_{\bar{v}}.
\end{split}
\end{equation}
Note that $T_{v}H^{*}_{\bar{u}}H_{\bar{u}}T_{\bar{v}}
=(H_{\bar{u}}T_{\bar{v}})^{*}H_{\bar{u}}T_{\bar{v}}$ is positive, must be self-adjoint.

It remains  to show that $T_{v}H^{*}_{\bar{u}}H_{\bar{u}}T_{\bar{v}}$ is an idempotent.
Since $v$ is analytic,
\begin{align*}
T_{v}H^{*}_{\bar{u}}H_{\bar{u}}
T_{\bar{v}}T_{v}H^{*}_{\bar{u}}H_{\bar{u}}T_{\bar{v}}
=&T_{v}H^{*}_{\bar{u}}H_{\bar{u}}
T_{|v|^{2}}H^{*}_{\bar{u}}H_{\bar{u}}T_{\bar{v}},
\end{align*}
let $|v|^{2}=uh+\bar{u}\bar{h}_{1}+1,h,h_{1}\in H^{2},$ for every $k$ in $K^{2}_{u},$ we have
\begin{align*}
H^{*}_{\bar{u}}H_{\bar{u}}T_{|v|^{2}}k
=&H^{*}_{\bar{u}}H_{\bar{u}}P(uh+\bar{u}\bar{h}_{1}+1)k\\
=&H^{*}_{\bar{u}}H_{\bar{u}}P(uhk+\bar{u}\bar{h}_{1}k+k)\\
=&H^{*}_{\bar{u}}H_{\bar{u}}(uhk+k)\\
=&k.
\end{align*}
Since $Range({H^{*}_{\bar{u}}H_{\bar{u}}})=K^{2}_{u},$
$T_{v}H^{*}_{\bar{u}}H_{\bar{u}}
T_{\bar{v}}T_{v}H^{*}_{\bar{u}}H_{\bar{u}}T_{\bar{v}}=
T_{v}H^{*}_{\bar{u}}H_{\bar{u}}
T_{\bar{v}}.$
\end{proof}
\begin{remark}
In fact, $\|h\|_{\infty}=1.$  Since $\ker Q$ is nontrivial,  there is a nonzero vector $x$ such that,
\begin{align*}
H^{*}_{\bar{f}}H_{\bar{f}}x=&H^{*}_{f}H_{f}x\neq0,\quad
\|H_{\bar{f}}x\|=\|H_{f}x\|
\end{align*}
and
$\|H_{f}x\|=\|S_{h}H_{\bar{f}}x\|\leq\|S_{h}\|\|H_{\bar{f}}x\|.$
Hence $\|S_{h}\|=\|h\|_{\infty}\geq1.$
\end{remark}

\begin{lemma}\label{SCC}
If $v\in H^{2}$ and $u$ is inner,
$|v|^{2}-1\in uH^{2}+\overline{uH^{2}}$
if and only if
there is a function $h\in H^{2}$ such that $|v|^{2}=Re(uh+1).$
\end{lemma}
\begin{proof}
Since
\begin{align*}
Re(uh+1)=&\frac{1}{2}(uh+1+\bar{u}\bar{h}+1)\\
=&u(\frac{1}{2}h)+\bar{u}(\frac{1}{2}\bar{h})+1,
\end{align*}
 $|v|^{2}=Re(uh+1)$ implies  $|v|^{2}-1\in uH^{2}+\overline{uH^{2}}.$

Suppose $|v|^{2}-1\in uH^{2}+\overline{uH^{2}},$ then there exist $F,G\in H^{2}$ such that $|v|^{2}-1=uF+\bar{u}\overline{G},$ and  $uF+\bar{u}\overline{G}$ is real-vauled,
$uF+\bar{u}\overline{G}=\bar{u}\overline{F}+uG.$
Hence,
\begin{align*}
u(F-G)=\bar{u}(\overline{F}-\overline{G}).
\end{align*}
The left-hand side of the above equation is analytic, the right-hand side is conjugate analytic, $u(F-G)$ is equals to a constant $\lambda.$  If $\lambda$ is not zero, then
$u\frac{1}{\lambda}(F-G)=1,$ and $u$ is outer \cite[6.20]{douglas2012banach}, that is a contradiction. Thus $\lambda=0,F=G,$ and $|v|^{2}=Re(u(2F)+1).$
\end{proof}
\begin{remark}
The set $\Theta=\{v:v\in H^{2},|v|^{2}-1\in uH^{2}+\overline{uH^{2}}\}$
is not empty.
It is easy to see that if $v$ is inner, $v\in \Theta.$
Using \eqref{QQ}, we have $Q=T_{v}T_{\bar{v}}-T_{uv}T_{\bar{u}\bar{v}},$ and the
range of $Q$ is $vH^{2}\ominus vuH^{2}=vK^{2}_{u}.$
Moreover, $u\pm1\in\Theta.$
\end{remark}

The following theorem summarizes Lemma \ref{K2}, Lemma \ref{SC} and Lemma \ref{SCC}.
\begin{theorem}\label{main3}
If $\varphi\in L^{\infty}$ then
$T^{*}_{\varphi}T_{\varphi}-T_{\varphi}T^{*}_{\varphi}$ is a nontrivial projection operator
if and only if  one of following conditions holds
\begin{enumerate}
\item The range of $T^{*}_{\varphi}T_{\varphi}-T_{\varphi}T^{*}_{\varphi}$
    is a model space, and
    $\varphi=a\theta+b\bar{\theta}+c,$ where $\theta$ is an inner function,$a,b$ and $c$ are constant with $|a|^{2}-|b|^{2}=1;$
\item The range of $T^{*}_{\varphi}T_{\varphi}-T_{\varphi}T^{*}_{\varphi}$
    is not a model space, and
$\varphi=uv+\bar{v}+c,$ where $u$ is inner, $c$ is constant, $v\in H^{2}$ with $|v|^{2}=Re(uh+1)(h\in H^{2})$.
\end{enumerate}
\end{theorem}

\section{Further discussion}
Now we study the $C^{*}-$algebra $\mathcal{T}_{u}$ generated by $\{T_{u}T_{\bar{u}}: u \text{ is an inner function}\}.$  Since the symbol mapping of every element in $\mathcal{T}_{u}$ is constant,$\mathcal{T}_{u}$
is a proper subalgebra of $\mathfrak{T}_{L^{\infty}}.$
The following theorem will give some information of the structure of $\mathcal{T}_{u}.$
\begin{theorem}
$\mathcal{T}_{u}$ is irreducible and contains all compact operators.
\end{theorem}
\begin{proof}
Suppose that $\mathcal{T}_{u}$ is reducible. Then there exists a nontrivial projection $E$  which commutes with each $T_{u}T_{\bar{u}}$ for all inner function $u.$  If $u$ is a M$\ddot{o}$bius transform
\[u=\varphi_{z}(w)=\frac{z-w}{1-\bar{z}w},\]and $k_{z}$ denote the normalized reproducing kernel at $z:$
$k_{z}(w)=\frac{\sqrt{1-|z|^{2}}}{1-\bar{z}w}.$
We have the following identity:
\begin{align}\label{kz}
I-k_{z}\otimes k_{z}=T_{\overline{\varphi_{z}}}T_{{\varphi_{z}}},
\end{align}
the identity can be found in \cite[p.480]{zheng1996distribution}.
Hence,
\begin{align*}
E(k_{z}\otimes k_{z})&=(k_{z}\otimes k_{z})E\\
(Ek_{z})\otimes k_{z}&=k_{z}\otimes (Ek_{z})\\
\langle Ek_{z},k_{z}\rangle Ek_{z}
&=\langle Ek_{z},Ek_{z}\rangle k_{z}\\
\|Ek_{z}\|^{2}Ek_{z}
&=\|Ek_{z}\|^{2}  k_{z}.
\end{align*}
If $Ek_{z}$ is not a zero vector, we have
$Ek_{z}= k_{z}.$ Thus unit disc $\mathbb{D}$ is the disjoint union of two sets, say $\mathbb{D}=\Sigma_{1}\cup\Sigma_{2},$ where
$\Sigma_{1}=\{z\in \mathbb{D}:Ek_{z}=0\}$ and
$\Sigma_{2}=\{z\in \mathbb{D}:Ek_{z}=k_{z}\}.$
So, at least one of $\Sigma_{1}$ and $\Sigma_{2}$ is
an uncountable set. at least of $\{k_{z}:z \in \Sigma_{1}\}$ and
$\{k_{z}:z \in \Sigma_{2}\}$ is dense in $H^{2}.$  Hence,
$E$ is zero operator or identical operator, which is a contradiction.
Using \eqref{kz}, we have $\mathcal{T}_{u}$ contains at least one nonzero compact operator. By \cite[5,39]{douglas2012banach}, $\mathcal{T}_{u}$ contains all compact operators.
\end{proof}

\bigskip \noindent{\bf Acknowledgement}.


\end{document}